\documentclass[reqno,11pt]{amsart}
\usepackage{amsmath, latexsym, amsfonts, amssymb, amsthm, amscd}
\usepackage[utf8]{inputenc}
\usepackage{graphics,epsf,psfrag,epsfig}

\numberwithin{equation}{section}

\newtheorem{theorem}{Theorem}[section]
\newtheorem{lemma}[theorem]{Lemma}
\newtheorem{proposition}[theorem]{Proposition}

\newtheorem{rem}[theorem]{Remark}

\renewcommand{\ge}{\geq}
\renewcommand{\le}{\leq}
\newcommand{\ind}{\mathbf{1}}

\newcommand{\Z}{\mathbb{Z}}

\newcommand{\bP}{{\ensuremath{\mathbf P}} }
\newcommand{\bE}{{\ensuremath{\mathbf E}} }

\DeclareMathSymbol{\leqslant}{\mathalpha}{AMSa}{"36} 
\DeclareMathSymbol{\geqslant}{\mathalpha}{AMSa}{"3E} 
\DeclareMathSymbol{\eset}{\mathalpha}{AMSb}{"3F}     
\renewcommand{\leq}{\;\leqslant\;}                   
\renewcommand{\geq}{\;\geqslant\;}                   
\newcommand{\dd}{\,\text{\rm d}}             

\newcommand{\bbE}{{\ensuremath{\mathbb E}} }

\newcommand{\bbN}{{\ensuremath{\mathbb N}} }

\newcommand{\bbP}{{\ensuremath{\mathbb P}} }

\newcommand{\bbZ}{{\ensuremath{\mathbb Z}} }

\newcommand{\ga}{\alpha}
\newcommand{\gb}{\beta}
\newcommand{\gga}{\gamma}            
\newcommand{\gd}{\delta}
\newcommand{\gep}{\varepsilon}       

\newcommand{\gO}{\Omega}

\makeatletter
\def\captionfont@{\footnotesize}
\def\captionheadfont@{\scshape}

\long\def\@makecaption#1#2{%
  \vspace{2mm}
  \setbox\@tempboxa\vbox{\color@setgroup
    \advance\hsize-6pc\noindent
    \captionfont@\captionheadfont@#1\@xp\@ifnotempty\@xp
        {\@cdr#2\@nil}{.\captionfont@\upshape\enspace#2}%
    \unskip\kern-6pc\par
    \global\setbox\@ne\lastbox\color@endgroup}%
  \ifhbox\@ne 
    \setbox\@ne\hbox{\unhbox\@ne\unskip\unskip\unpenalty\unkern}%
  \fi
  \ifdim\wd\@tempboxa=\z@ 
    \setbox\@ne\hbox to\columnwidth{\hss\kern-6pc\box\@ne\hss}%
  \else 
    \setbox\@ne\vbox{\unvbox\@tempboxa\parskip\z@skip
        \noindent\unhbox\@ne\advance\hsize-6pc\par}%
\fi
  \ifnum\@tempcnta<64 
    \addvspace\abovecaptionskip
    \moveright 3pc\box\@ne
  \else 
    \moveright 3pc\box\@ne
    \nobreak
    \vskip\belowcaptionskip
  \fi
\relax
}
\makeatother
\def\writefig#1 #2 #3 {\rlap{\kern #1 truecm
\raise #2 truecm \hbox{#3}}}

\newcommand{\var}{{\rm Var}}
\newcommand{\Tm}{T_{\rm mix}}
\newcommand{\Tr}{T_{\rm rel}}

\title[Cutoff for SEP on the complete graph]{Cutoff phenomenon for the simple exclusion process on the complete graph}

\begin{document}

\author{Hubert Lacoin}
\address{CEREMADE Université Paris Dauphine,
Place du Maréchal De Lattre De Tassigny,
75775 PARIS CEDEX 16 - FRANCE }
\email{lacoin@ceremade.dauphine.fr}

\author{Rémi Leblond}
\address{Ecole Polytechnique, D\'epartement de Math\'ematiques Appliqu\'ees, 91128 Palaiseau, France}
\email{remi.leblond@polytechnique.edu}

\begin{abstract}
We study the time that the simple exclusion process on the complete graph needs to reach equilibrium in terms of 
total variation distance. For the graph with $n$ vertices and $1\ll k<n/2$ particles, we show that the mixing time is of order
$\frac{1}{2}n\log \min( k, \sqrt{n})$, and that around this time, 
for any $\gep$, the total variation distance drops from $1-\gep$ to $\gep$ in a time window of whose width is of order $n$ 
(i.e.\ in a much shorter time). Our proof is purely probabilistic and self-contained.
\end{abstract}

\maketitle

\section{Introduction}

Let $G=(V,E)$ be a finite connected graph and $1\le k \le  |V|-1$ an integer. We define a \textit{ configuration } as an element of 
$\eta\in {\{0,1\}}^V$ with $k$ ones and $|V|-k$ zeros (ones can be considered as particles moving on the graph). 
The simple exclusion process on the graph $G$ with $k$ particles, can be described as follows: One starts with a given configuration 
and at each time step, one chooses an edge $e$ uniformly at random in $E$ and one interchanges the contents (zero or one) 
of the two vertice adjacent to $e$.
This Markov chain is reversible and has the uniform measure over all configurations as equilibrium measure.

In this paper we study the rate of convergence to equilibrium of the exclusion process on the complete graph with $n$ vertices, 
$k(n)$ particles, and $k$ going to infinity with $n$ . For obvious symmetry reasons, one can restrict the problem to 
the case $k\le n/2$ without any loss of generality.
For the sake of clarity we give a formal definition of the exclusion process

\medskip

The configuration space is 
\begin{equation}
 \gO(n,k(n)):=\Big\{ \eta \in \{0,1\}^{[1,n]\cap \bbN} \ | \ \sum_{x\in [1,n]\cap \bbN} \eta(x)=k(n)\Big\}.
\end{equation}
We consider the discrete-time Markov chain on $\gO(n,k(n))$ that at each time step, independently selects two vertice uniformly at random in 
$ \{1, \dots , n\}$ and interchanges their contents (emptiness or particle). Note that with probability $1/n$ the two chosen vertices 
are the same: in that case nothing happens.
We give now the transition kernel of this process:
The symmetric group $\mathcal S_n$ acts transitively on  $\gO(n,k(n))$ in a natural way. If $\sigma \in \mathcal S_n$  then
\begin{equation}
 (\sigma.\eta)(x):=\eta(\sigma^{-1}(x)),
\end{equation}
and for $\eta, \eta'\in \gO(n,k(n))$, $\eta\neq \eta'$, the transition rates are given by
\begin{equation}
 P(\eta, \eta')=\left\{
    \begin{array}{ll}
        \frac{2}{n^2} & \text{ if }  \eta'=\tau.\eta \text{ for some transposition } \tau \\
        0 & \text{if it is not the case}
    \end{array}\right.
\end{equation}
One can check that this implies $ P(\eta, \eta) \ge 1/2$.

This Markov chain is reversible, aperiodic and its equilibrium measure is the uniform measure on $\gO(n,k(n))$ that we denote by $\pi$.
Given an initial configuration $\xi\in \gO(n,k(n))$, one writes $\bbP^{\xi}$ (and $\bbE^{\xi}$ denotes the associated expectation) for the
law of the Markov chain $(\eta_t)_{t\ge 0}$ started from the configuration $\eta_0=\xi$ and $\mu_t^{\xi}$ for the marginal distribution of
$\bbP^{\xi}$ at time  $t$ ($t\in \bbN$).

We study the convergence to equilibrium of this chain. Distance to equilibrium is given by the following quantity
\begin{equation}
 d^{(n)}(t):= \max_{\xi\in \gO(n,k(n))} \| \mu^{\xi}_t-\pi\|.
\end{equation}
where $\|\cdot\|$ denotes the total variation distance: for two measures on $\gO(n,k(n))$ 
\begin{equation}
  \| \mu-\pi\|:= \frac{1}{2}\sum_{\eta\in \gO(n,k(n))} |\mu(\eta)-\pi(\eta)|.
\end{equation}

The main result of this paper is a sharp estimate of the time needed to reach equilibrium.

\begin{theorem}\label{mainres}
If $\lim_{n\to \infty } k(n)/\sqrt{n}= \infty$, then for every $\gep>0$ there exists $\gb>0$ such that for all $n$

\begin{equation}\begin{split}\label{cokkk}
d^{(n)}\left(\frac{1}{4}n\log n+\gb n\right)&\le \gep\\
d^{(n)}\left(\frac{1}{4}n\log n-\gb n\right)&\ge 1-\gep.
\end{split}\end{equation}

\medskip

If $\lim_{n\to \infty } k(n)/\sqrt{n}=0 $,  and $\lim_{n\to \infty} k(n)= \infty$, then for every $\gep>0$ there exists $\gb>0$ such that
for all $n$
\begin{equation}\begin{split}\label{cokk}
d^{(n)}\left(\frac{1}{2}n\log k(n)+\gb n\right)&\le \gep\\
d^{(n)}\left(\frac{1}{2}n\log k(n)-\gb n\right)&\ge 1-\gep.
\end{split}\end{equation}

\medskip

If $k(n)/\sqrt{n}\to l\in(0,\infty)$, then \eqref{cokkk} and \eqref{cokk} hold.

\end{theorem}
%

The function $t\mapsto d^{(n)}(t)$ is non-increasing. Thus
for any $\gep$ one can set
\begin{equation}
\Tm^{(n)} (\gep):= \inf\{ t \ | \ d^{(n)}(t)\le \gep\}=\sup\{ t \ | \ d^{(n)}(t)> \gep\}.
\end{equation}

From the above theorem one has that for any $\gep>0$
\begin{equation}
\Tm^{(n)}(\gep)-\Tm^{(n)}(1-\gep)=O_{\gep}(n)=o_{\gep}(\Tm),
\end{equation}
where $O_{\gep}$ and $o_{\gep}$ underline dependence in $\gep$.
In words: the time in which the distance to equilibrium drops from close to one to close to zero is at most of order $n$ and much smaller than $\Tm$.
This phenomenon is known as {\it cutoff}.
It was first identified by Diaconis and Shashahani \cite{DS81} for the random walk on the symmetric group generated by transpositions 
(see also \cite{BSZ} for a recent extension of this result with a probabilistic proof), 
and was given its name in the celebrated paper of Aldous and Diaconis \cite{AD} where it is shown that cutoff occurs for top-to-random card shuffle.
The bound that is obtained for $\Tm^{(n)}(\gep)-\Tm^{(n)}(1-\gep)$ (in the present case, $O(n)$) is often called the cutoff window . Our result is optimal in the sense that $O(n)$ is the best window one can obtain:
it will be shown in the proof that
\begin{equation}\label{optimo}
 \lim_{\gep\to 0} \liminf_{n\to \infty}\frac{\Tm^{(n)}(\gep)-\Tm^{(n)}(1-\gep)}{n}=\infty.
\end{equation}

\medskip

The simple exclusion process on the complete graph maps to another problem: the Bernouilli Laplace Diffusion Process. 
In \cite{DS87}, Diaconis and Shashahani studied this model and, using purely algebraic methods, proved cutoff in the case $k(n)=n/2$. 
Their method should be extendable to some other values of $k$ (e.g.\ using the same method Donnelly, Lloyd and Sudbury \cite{DLS}
extended the result to the case where $k=\alpha n$ for some $\alpha\in(0,1)$), but it clearly fails to give the right result when $k(n)\ll\sqrt{n}$ 
(e.g.\ the upper-bound given in Theorem $2$ fails to be sharp in that particular case). 
We also underline that the methods we present here are purely probabilistic.

\medskip

The simple exclusion process on the complete graph can be seen as a projection of the random walk on the symmetric group generated by 
transposition, and therefore the mixing time for simple exclusion is always smaller than the mixing time for random transposition.
What our result underlines is that while the spectral gaps (this a general result that holds for every graph, see \cite{CLR}) for the two processes are the same, the mixing time differ:
the mixing time for the random transposition model is $n/2\log n(1+o(1))$, whereas the mixing time for the exclusion process is at most
 $n/4 \log n(1+o(1))$. This result is specific to the complete graph:
for the exclusion process on the segment or on the circle, mixing time for exclusion process with a density of particle and 
interchange process are expected to coincide \cite{W}.

\medskip

Let us also compare the mixing time of the simple exclusion process with the mixing time of the simple exclusion process with $k$ 
labeled particles: the space of configurations is 
\begin{equation}
 \gO'(n,k(n)):=\left\{ \eta \in \{0,1,\dots,k\}^{[1,n]\cap \bbN} \ | \ \forall i\in\{1,\dots,k\}\ \exists ! x, \eta(x)=i \right\}.
\end{equation}
The rules for the evolution are the same: at each time step, one chooses two vertice at random and interchanges their contents. 
As there is no risk of confusion we use for the labeled process the same notation as that for the exclusion process. The equilibrium measure $\pi$ for this process is the uniform measure over
$\gO'(n,k(n))$. 
What we can show is that if $k\ll\sqrt{n}$ then the mixing times of labeled and unlabeled exclusion process coincide but that they differ as soon as $k\gg\sqrt{n}$.
We suspect that for every value of $k$ and every $\gep$ one has
\begin{equation}
 \Tm'(\gep):= \frac{n}{2}\log k+O(n).
\end{equation}
We prove in the following that
\begin{theorem}\label{auxres}
 For the exclusion process on the complete graph with $n$ vertice and $k$ labeled particles, for every $\gep$, there exists  
$\gb>0$ such that for every $k$ and $n$
\begin{equation}
 d^{(n)}\left(\frac{1}{2}n\log k-\gb n\right)\ge 1-\gep.
\end{equation}
Moreover if $\lim_{n\to \infty} k(n)/\sqrt{n}=0$ then for every $\gep$, there exist 
$\gb>0$ such that for all $n$
\begin{equation}
 d^{(n)}\left(\frac{1}{2}n\log k(n)+\gb n\right)\le \gep.
\end{equation}
\end{theorem}

The mixing time of the simple exclusion process has been studied for some other graphs than the complete graph.
However, to our knowledge, cutoff has not been proved for any other graph.
We refer to \cite{W} for a study of the mixing time of the simple exclusion process on the segment $\{0,\dots,n\}$ 
(the edges of the graph are the $(k,k+1)$, $k\in [0,n-1]$), \cite{M} for simple exclusion on the $d$-dimensional torus, and \cite{Ol}
for a recent general study of the  exclusion process mixing time.

\medskip

The sequel of the paper is organized as follows
\begin{itemize}
 \item In Section \ref{badc}, we reduce the study of the the unlabeled exclusion process to the study of a birth of death chain, 
which is a first step towards the proof of Theorem \ref{mainres}.
 \item In Section \ref{petitk}, we prove Theorem \ref{mainres} in the case of small $k$.
 \item In Section \ref{Labelpro}, we prove Theorem \ref{auxres}.
 \item In Section \ref{grandk}, we prove Theorem \ref{mainres} in the case of large $k$.
\end{itemize}

\section{Reduction to the study of a birth and death chain}\label{badc}

Our Markov chain is a lazy simple random walk on a transitive graph. Therefore, by transitivity, the distance $ \| \mu^{\xi}_t-\pi\|$ does not depend on the initial configuration $\xi$. We can set  $\eta_0$ to be 
\begin{equation}
 (\eta_0)(x):=\left\{
    \begin{array}{ll}
       1 & \text{ if }  x\in [1, k(n)] \\
       0 & \text{if not}
    \end{array}\right.
\end{equation}
and we simply write $\bbP$ - and $\bbE$ for the associated expectation - (resp.\ $\mu_t$) for the law of $(\eta_t)_{t\ge 0}$ (resp.\ $\eta_t$) starting from this configuration.

\medskip

We now claim that for every $t$, $\mu_t$ is invariant under permutations of the coordinates in $\{1,\dots, k(n)\}$
and in $\{k(n)+1,\dots, n\}$. This is obviously true for $t=0$, and this remains true for $t> 0$ as the dynamic itself is invariant under these permutations.
Therefore, if one sets
\begin{equation}\begin{split}
 W(\eta)&:=\sum_{x=1}^{k(n)} \eta(x),\\
 W_t   &:= W(\eta_t).
\end{split}\end{equation}
then for any $m$ and $t$ such that $\mu(W_t=m)>0$,
$\mu_t(\cdot|W=m)$ is the uniform measure over all the configuration $\eta$ such that $W(\eta)=m$.

Let $\bar \mu_t$ and $\bar \pi$ be the law of $W$ under $\mu_t$ and $\pi$ respectively.
The preceding remarks imply that
\begin{equation}
 d(t)= \|\bar \mu_t-\bar \pi_t \|
\end{equation}
One can check that the evolution of $W_t$ is Markovian. Our problem is now confined to the study of the mixing time of this new Markov chain,
which is what is called  a {\sl birth and death chain} on $\{0,\dots, k(n)\}$.

We write $\bar \mu^i_t$ for the law of $W_t$ starting from $W_0=i$. As $(W_t)_{t\ge 0}$ is a projection of the Markov chain $(\eta_t)_{\ge 0}$, one has
\begin{equation}
\|\bar \mu_t-\bar \pi\|=\max_{\xi\in \gO(n,k(n))}\|\mu_t^{\xi}-\pi\|\ge \max_{i\in\{1,\dots,k\}}\|\bar \mu^i_t-\bar \pi\|\ge \|\bar \mu_t-\bar \pi\|,
\end{equation}
that is to say that for any value of $t$,  $W_t$ is farther from the equilibrium measure if $W_0=k$.

\begin{rem}\rm
 At this point of our analysis, one can already show that there is cutoff for our process. 
Indeed, the cutoff phenomenon for general birth and death chain has been studied in \cite{DLP}, 
in which the authors prove that $\Tr=o(\Tm(1/4))$ is a necessary and sufficient for having cutoff (where $\Tr$, the relaxation time
 is by definition the inverse of the spectral gap).
This condition can be checked rather easily in our case. However, one cannot get the location of the cutoff,
nor the correct order for the size of the window by using only this general result. 
\end{rem}

We use the notation $\bar P$ to denote the transition probability of $(W_t)_{t\ge 0}$. For the sake of clarity, we often omit the dependence in $n$ in the notation.
We have
\begin{equation}\begin{split}
 \bar P (i,i+1)&= \frac{2 (k-i)^2}{n^2},\\
 \bar P (i,i-1)&= \frac{2 i(n-2k+i)}{n^2},\\
 \bar P (i,i)&=   \frac{n^2-2[(k-i)^2+i(n-2k+i)]}{n^2}.
\end{split}\end{equation}

We end this section with a first simple Lemma giving the expectation of $W_t$. It will constantly be used in the sequel.

\begin{lemma}\label{expec}
 One has, for any value of $k$, for any $t$
\begin{equation}
 \bbE^{\xi}[W_t]=\left(W(\xi)-\frac{k^2}{n}\right) \left(1-\frac{2}{n}\right)^t+ \frac{k^2}{n}
\end{equation}

\end{lemma}
\begin{proof}
Using the jump rates we compute the expected value of $W_{t+1}$ given $W_t$.
One has
\begin{equation}
 \bbE^{\xi}[W_{t+1}|W_t]=W_t+ \bar P (W_t,W_t+1)-\bar P (W_t,W_t-1)= \frac{k^2}{n}+\left(W_t-\frac{k^2}{n}\right)\left(1-\frac{2}{n}\right).
\end{equation}
Taking the expectation on both sides, and making a trivial induction, one gets the desired result. 
\end{proof}

\section{The case $k(n)\ll n^{1/2}$} \label{petitk}

In this section we prove the main theorem with the assumptions that $\lim_{n\to \infty} k(n)/\sqrt{n}=0$ and $\lim_{n\to \infty} k(n)=\infty$.

\subsection{Upper bound on $\Tm$}\label{TMM}

Under the assumption that $\lim_{n\to\infty} \frac{k(n)}{\sqrt{n}}=0$, $\bar \pi(W=0)=1-o(1)$. Indeed, with this condition, the expectation of $W$ at equilibrium is
\begin{equation}
 \bar \pi(W)= \frac{k^2(n)}{n}=o(1).
\end{equation}

Let us choose $\gamma>0$ and set $t_{\gamma}=\frac{n}{2}\log k(n) + \gamma n$ (is has to be thought as the integer part, but we omit this in the notation to keep things simpler; at any rate it would not change the proof).
One has 
\begin{equation}
 \bar \mu_{t_{\gga}} (W)=k\left(1-\frac{2}{n}\right)^{\frac{n}{2}\log k + \gamma n}+o(1)
\end{equation}
Altogether we get that
\begin{equation}
\| \bar \mu_{t_{\gga}}-\bar \pi\|\le e^{-2\gamma}+o(1).
\end{equation}
And therefore
\begin{equation}
 \Tm(\gep)\le \frac{n}{2}\left(\log k(n) - \log \gep +o(1)\right). 
\end{equation}
\qed

\subsection{Lower bound on $\Tm$}\label{TMMMM}

To get the other bound, we make the following consideration:
the equilibrium measure $\bar\pi$ is concentrated on the event $\{W=0\}$.
Therefore, on the original exclusion process, every particle has to be moved at least once in order to be significantly close to equilibrium.
To formalize this properly, we present an alternative construction of the simple exclusion process.

Let $(X_t,Y_t)_{t \ge 1}$ be a sequence of i.i.d.\ random variables distributed uniformly on $\{1,\dots,n\}^2$ (we include this process in the probability law $\bbP$).
Under $\bbE^{\xi}$, we start from $\eta_0= \xi$ and we build $\eta_t$ from $\eta_{t-1}$ by interchanging the content of sites $X_t$ and $Y_t$ if $X_t\ne Y_t$:

\begin{equation}
 \eta_t(x):=\left\{
    \begin{array}{ll}
        \eta_{t-1}(x)& \text{ if }  x\notin  \{X_t,Y_t\}, \\
        \eta_{t}(Y_t)& \text{ if }  x=X_t,\\
	\eta_{t}(X_t)& \text{ if }  x=Y_t. 
    \end{array}\right.
\end{equation}

We define $\tau$ as the time were all the sites in $\{0,\dots,k(n)\}$ have been selected at least once by the process $(X,Y)$
\begin{equation}
 \tau:=\inf\left \{t\ge 0\ | \bigcup_{s=1}^t \{X_s,Y_s\}\supset \{1,\dots, k(n)\}\right\}.
\end{equation}
Notice that if $t<\tau$, then $W_t\ne 0$, so that 
\begin{equation}
\|\bar \mu_t-\gd_{W=0}\|\ge \bbP\left[\tau> t\right]
\end{equation}
Therefore one has
\begin{equation}\label{taut}
 \|\bar \mu_t-\bar \pi\|\ge \bbP\left[\tau> t\right]-o(1) 
\end{equation}

Estimating the time $\tau$ boils down to the so-called  coupon collector problem (see \cite{LPW},  Section 2.2 in particular).
Set $X'_{2s-1}:=X_s$ and $X'_{2s}=Y_s$. Then $X'_s$ is an i.i.d.\ sequence
One has the following equality in law
\begin{equation}\label{taut2}
 \tau=\lceil\tau'/2\rceil=\left\lceil \frac{1}{2}\sum_{i=1}^k \mathcal E_i\right\rceil.
\end{equation}
where
\begin{equation}
 \tau':=\inf\left \{t\ge 0\ | \{X'_s, s=1,\dots, t \}\supset \{1,\dots, k(n)\}\ \right\}.
\end{equation}
and the $\mathcal E_i$ are defined by
\begin{equation}
 \sum_{j=1}^i \mathcal E_j := \inf\left \{t\ge 0\ \big| \#(\{X'_s | s=1,\dots,t\}\cap [1,k(n)])=i \right\}.
\end{equation}
It is not difficult to check that $(\mathcal E_i)_{i\in [1,k]}$  are independent geometric variables of mean $(\frac{n}{k-i+1})_{i\in[1,k]}$. From this, one gets the following moment estimates: for some constant $C$
\begin{equation}\begin{split}
 \bbE\left[\sum_{i=1}^k \mathcal E_i\right]&\ge n\log k-Cn ,\\
 \var_{\bbP}\left[\sum_{i=1}^k \mathcal E_i\right]&= \sum_{i=1}^k \frac{1-(i/n)}{(i/n)^2}\le Cn^2.
\end{split}\end{equation}
Therefore, if one chooses $t_{\gamma}:=\frac{1}{2}n\log k-\gamma n$ (suppose that this in an integer), one has by classical second moment inequality:
\begin{equation}
 \bbP\left[ \sum_{i=1}^k \mathcal E_i \le 2t_{\gamma}\right]\le \frac{C}{(2\gamma-C)^2}
\end{equation}
And therefore from \eqref{taut} and \eqref{taut2}
\begin{equation}
 d(t_{\gamma})\ge 1- \frac{C}{(2\gamma-C)^2}-o(1),
\end{equation}
and hence, for any $\gep>0$
\begin{equation}
  \Tm(1-\gep)\ge \frac{n}{2}\left(\log k(n)-\sqrt{C/\gep}-C+o(1)\right).
\end{equation}
(One could get a tighter bound with $\log \gep$ instead of $-\gep^{-1/2}$ by using exponential moments instead of second moment).
\qed

\section{Bounds for the labeled process}\label{Labelpro}
The methods of the previous section can be applied for the proof of Theorem \ref{auxres}.

For the lower-bound we remark that at equilibrium, for any $i\in \{1,\dots,k\}$, $\pi(\eta(i)=i)=\frac{1}{n}$. Thus in every case the expected number 
of fixed points is less than one:  $\pi(\#\{i | \eta(i)=i\})\le 1$, and for any integer $K$
\begin{equation}
 \pi(\#\{i | \eta(i)=i\}\ge K)\le 1/K.
\end{equation}
One constructs the exclusion process from $(X_t,Y_t)_{t\ge 0}$ as in the previous section. 
We define $\tau$ as the first time at which all of the $k$ first sites have been selected.
\begin{equation}
 \tau:=\inf\left \{t\ge 0\ |\ \#((\bigcup_{s=1}^t \{X_s,Y_s\})\cap\{1,\dots, k(n)\})\ge k-K \right\}.
\end{equation}
If one starts the process from
\begin{equation}
 (\eta_0)(x):=\left\{
    \begin{array}{ll}
       x & \text{ if }  x\in [1, k(n)] \\
       0 & \text{if not}
    \end{array}\right.
\end{equation} 
One has
\begin{equation}
 \mu_t^{\eta_0}(\#\{i | \eta(i)=i\}\ge K)\ge \bP(t>\tau)
\end{equation}
and therefore
\begin{equation}
 \|\mu_t^{\eta_0}-\pi\|\ge \bP(t>\tau)-\frac{1}{K}.
\end{equation}
Taking the same definition for $\mathcal E_i$ as in the previous section one has
\begin{equation}
 \tau= \lceil \sum_{i=1}^{k-K}\mathcal E_i \rceil, 
\end{equation}
and $\mathcal E_i$ are independent geometric variables of respective mean $\left(\frac{n}{k-i-n}\right)$.
One has:
\begin{equation}\begin{split}
 \bbE\left[\sum_{i=1}^{k-K} \mathcal E_i\right]&\ge n\log k-n\log K-Cn ,\\
 \var_{\bbP}\left[\sum_{i=1}^k \mathcal E_i\right]&= \sum_{i=K}^{k} \frac{1-(i/n)}{(i/n)^2}\le Cn^2/K,
\end{split}\end{equation}
and therefore, using Chebychev inequality one gets:
\begin{equation}
 P(\tau \le \frac{n}{2}\log k -n \log K)\le \frac{C}{K(\log K-C)^2}.
\end{equation}
Overall, this gives that for $K$ sufficiently large and for $t= \frac{n}{2}\log k -n \log K$:
\begin{equation}
  \|\mu_t^{\eta_0}-\pi\|\ge 1-\frac{2}{K}.
\end{equation}
which gives
\begin{equation}
 \Tm(1-\gep)\ge \frac{n}{2}\log k + n\log \gep/2.  
\end{equation}

For the upper bound, we assume that $\lim_{n\to \infty} k(n)/\sqrt{n}=0$. We notice that as for the unlabeled process, the distance to 
equilibrium 
is the same for every starting position by symmetry.
\begin{equation}
d(t)=\|\mu_t^{\eta_0}-\pi\|.
\end{equation}
Let $W$ be the number of particle lying on the vertices $\{1,\dots,k\}$. ($W(\eta):=\sum_{x=1}^k \ind_{\eta(x)\ne 0}$).
Now notice that for every $t\ge k$ one has
\begin{equation}
\mu_t^{\eta_0}(\cdot |W=0)= \pi(\cdot |W=0),
\end{equation}
as the initial condition and the dynamics are invariant under permutation of $\{k(n)+1,\dots,n\}$.
Therefore the same analysis as in section \ref{TMM} gives
\begin{equation}
 \Tm(\gep)\le \frac{n}{2}\left(\log k(n) - \log \gep +o(1)\right). 
\end{equation}

\section{The cases $k(n)\gg\sqrt{n}$ and $k(n)\approx \sqrt{n}$}\label{grandk}

In this section we prove the main theorem with the assumption that either\\
$\lim_{n\to \infty} {k(n)}/{n}=\infty$ or
$\lim_{n\to \infty} {k(n)}/{n}=l>0$. The latter case is a bit more complicated than the other as the distribution of $W$ at
equilibrium is asymptotically non-degenerate.
One makes use of second moment arguments for the lower bound of the mixing time, and  a diffusion argument to get an upper bound.

\subsection{Lower bound on the mixing time}
In this section we work with the weaker assumption $\liminf_{n\to \infty} k(n)/\sqrt{n}>0$.
To get the right bound on the mixing time one uses a second moment method.
The first and essential step is to compute a tight estimate of $\bar\mu_t\left[W^2 \right]$.
We start by writing an explicit formula for the second moment of $W_t$.
\begin{lemma}
\begin{multline}
 \bbE[W_t^2]= \left(1-\frac{4}{n}+\frac{4}{n^2}\right)^t k^2\\+
\sum_{s=0}^{t-1}\left[\left(\frac{4k^2}{n^2}-\frac{8k}{n^2}+\frac{2i}{n}\right)\bbE[W_s]+\frac{2k^2}{n}\right] \left(1-\frac{4}{n}+\frac{4}{n^2}\right)^{t-1-s}.
\end{multline}
\end{lemma}

\begin{proof}
 One simply uses the transition of the Markov chain to get a recurrence relation
\begin{equation}
 \bbE[W_{t+1}^2|W_t]= W_t^2+ 2W_t(\bar P(W_t,W_t+1)-\bar P(W_t,W_t-1))+\bar P(W_t,W_t+1)+\bar P(W_t,W_t-1).
\end{equation}
By taking the expectation on both sides, one gets the following recursive relation
\begin{equation}
 \bbE(W^2_{t+1})=\left(1-\frac{4}{n}+\frac{4}{n^2}\right)\bbE[W_t^2]+ \left(\frac{4k^2}{n^2}-\frac{8k}{n^2}+\frac{2i}{n}\right)\bbE [W_t]+\frac{k^2}{n}.
\end{equation}
which after induction gives the expected result.
\end{proof}

Then, using the formula above, we get a clean bound on the variance of $W_t$.

\begin{lemma}
 There exists a constant $C$ such that for any $K$, $n$ large enough (depending on $K$), and $
t= \frac{n}{4}\log n-\frac{\gga n}{2}$, $\gga\in [-K,K]$
\begin{equation}
 \var_{\bbP} (W_t^2)\le \frac{Ck^2}{n}+\frac{e^{\gga}k}{\sqrt{n}}.
\end{equation}
\end{lemma}
\begin{proof}
Suppose that $t=\frac{n}{4}(\log n-2\gamma)$ for some $\gamma \in [-K,K]$. All the $O(\cdot)$ are uniform in $\gamma \in [-K,K]$ when $n$ is large enough.
We start by giving an estimate of the expectation squared
\begin{multline}
 [\bbE(W_t)]^2= \left[\frac{k^2}{n}+\left(k-\frac{k^2}{n}\right)\left(1-\frac{1}{2n}\right)^t\right]^2\\
=\left[\frac{k^2}{n}+\left(k-\frac{k^2}{n}\right)\frac{e^{\gamma}}{\sqrt{n}}\left(1+O\left(\frac{\log n}{n}\right)\right)\right]^2
\\=\frac{k^4}{n^2}+\frac{2k^2}{n^{3/2}}\left(k-\frac{k^2}{n}\right)e^{\gamma}+n^{-1}\left(k-\frac{k^2}{n}\right)^2e^{2\gamma}+O(k^2/n).
\end{multline}
Estimating $\bbE(W_{t}^2)$ is a bit more tricky. For practical reasons we divide it in three terms
\begin{multline}
 \bbE(W_{t}^2)=\left(1-\frac{4}{n}+\frac{4}{n^2}\right)^t k^2
+\frac{2k^2}{n^2}\sum_{s=0}^{t-1}\left(1-\frac{4}{n}+\frac{4}{n^2}\right)^{t-1-s}\\
+\sum_{s=0}^{t-1}\left(\frac{4k^2}{n^2}-\frac{8k}{n^2}+\frac{2}{n}\right) \bbE(W_s)\left(1-\frac{4}{n}+\frac{4}{n^2}\right)^{t-1-s}
\end{multline}
The first one gives
\begin{equation}
\left(1-\frac{4}{n}+\frac{4}{n^2}\right)^t k^2=\frac{k^2e^{2\gamma}}{n}+O(k^2/n).
\end{equation}
The second:
\begin{equation}
 \frac{2k^2}{n^2}\sum_{s=0}^{t-1}\left(1-\frac{4}{n}+\frac{4}{n^2}\right)^{t-1-s}=O(k^2/n).
\end{equation}
We divide the third term it into two contributions
\begin{multline}
\sum_{s=0}^{t-1}\left(\frac{4k^2}{n^2}-\frac{8k}{n^2}+\frac{2}{n}\right)\bbE(W_s)\left(1-\frac{4}{n}+\frac{4}{n^2}\right)^{t-1-s}\\
= \left(\frac{4k^2}{n^2}-\frac{8k}{n^2}+\frac{2}{n}\right)\left(k-\frac{k^2}{n}\right)\left(1-\frac{4}{n}+\frac{4}{n^2}\right)^{t-1}
\sum_{s=0}^{t-1} \left(\frac{1-\frac{2}{n}}{1-\frac{4}{n}+\frac{4}{n^2}}\right)^s\\
+\left(\frac{4k^2}{n^2}-\frac{8k}{n^2}+\frac{2}{n}\right)\frac{k^2}{n}\sum_{s=0}^{t-1}\left(1-\frac{4}{n}+\frac{4}{n^2}\right)^{s}.
\end{multline}
The first contribution can be estimated as follows
\begin{multline}
 \left(\frac{4k^2}{n^2}+\frac{2}{n}+O(k/n^2)\right)\left(k-\frac{k^2}{n}\right)\frac{e^{2\gga}}{n}\left(1+O\left(\frac{\log n}{n}\right)\right)\frac{\sqrt{n}e^{-\gga}-1+O(n^{-1/2}\log n)}{\frac{2}{n}+O(n^{-2})}\\
=\frac{2k^2}{n^3/2}\left(k-\frac{k^2}{n}\right)e^{\gga}-\frac{2k^2}{n^{2}}\left(k-\frac{k^2}{n}\right)e^{2\gga}+\frac{e^{\gga}k}{\sqrt{n}}+O(k^2/n).
\end{multline}
The second is equal to
\begin{equation}
 \left(\frac{4k^2}{n^2}+O(1/n)\right)\frac{k^2}{4}\left(1-\frac{e^{2\gga}}{n}+O\left(\frac{1}{n}\right)\right)=\frac{k^4}{n^2}-e^{2\gga}\frac{k^4}{n^3}+O(k^2/n).
\end{equation}
Summing everything up gives the expected result.
\end{proof}

Now we use the bounds that we have on the second and first moment to bound the mixing time.
Set $t=\frac{n}{4}(\log n-2\gamma)$.
Let $(W_1, W_2)$ be a maximal coupling between $\bar \mu_t$ and $\bar \pi$, and let $\bar \nu_t$ be its law (such that $\|\bar\mu_t-\bar \pi\|=\bar\nu_t(W_1\ne W_2)$).

One has
\begin{multline}
  \bar \nu_t(W_1-W_2)^2=\bar \nu_t\{W_1\ne W_2\} \bar \nu_t\left((W_1-W_2)^2\ |\ W_1\ne W_2\right)\\
\ge \bar \nu_t\{W_1\ne W_2\}  \left[\bar\nu_t\left((W_1-W_2)\ |\ W_1\ne W_2\right)\right]^2=\frac{\left(\nu_t(W_1-W_2)\right)^2}{\|\bar\mu_t-\bar \pi\|}.
\end{multline}
And hence
\begin{equation}
 \|\bar\mu_t-\bar \pi\| \ge  \frac{[\nu_t(W_1-W_2)]^2}{\bar \nu_t(W_1-W_2)^2}=\frac{1}{1+\frac{\var_{\bar \nu_t}(W_1-W_2)}{[\nu_t(W_1-W_2)]^2}}.
\end{equation}
It is then easy to compute the moments. The expectation is given by
\begin{equation}
 \bar \nu_t (W_1-W_2)=\bar\mu_t( W)-\bar\pi( W) = e^{\gga}\frac{1}{\sqrt{n}}\left(k-\frac{k^2}{n}\right)(1+o(1)).
\end{equation}
For the variance, first notice that:
\begin{multline}
 \var_{\bar \pi}(W)=\sum_{i,j=1}^k \bar \pi ( \eta(i) \eta(j))-\frac{k^4}{n^2}= k\bar\pi(\eta(1))+ k(k-1)\bar \pi (\eta(1)\eta(2))-\frac{k^4}{n^2}\le \frac{k^2}{n}.
\end{multline}
And therefore
\begin{equation}
  \var{\bar \nu_t}(W_1-W_2)^2\le 2 [\var_{\bar \mu_t}(W)+\var_{\bar \pi}(W)]
 \le C' \frac{k^2}{n}+\frac{e^{\gga}k}{\sqrt{n}}.
\end{equation}
Hence one gets 
\begin{equation}
 \|\bar\mu_t-\bar \pi\|\ge \left(1+\frac{C' \frac{k^2}{n}+\frac{e^{\gga}k}{\sqrt{n}}}{\frac{e^{2\gga}}{n}\left(k-\frac{k^2}{n}\right)^2
(1+o(1))}\right)^{-1}\!\!\!
\ge \! \left(1+4C'e^{-2\gga}+\frac{4\sqrt{n}}{k}e^{-\gga}+o(1)\right)^{-1}\!\!\!\!.
\end{equation}
where in the last line one used $k\le \frac{n}{2}$.
Using the assumption that $\liminf k(n)/\sqrt {n}>0$, one obtains that for $n$ large enough
\begin{equation}
  \|\bar\mu_t-\bar \pi\|\ge \left(1+C''\max(e^{-2\gga},e^{-\gga})+o(1)\right)^{-1}
\end{equation}
Therefore
\begin{equation}
\Tm(\gep)\ge \left\{\begin{array}{ll}\frac{n}{4}\log n-\frac{n}{4}\log \left(\frac{C''}{\gep^{-1}-1}\right) &\text{ if } \gep\ge (1+C'')^{-1}\\
  \frac{n}{4}\log n+\frac{n}{2}\log \left(\frac{\gep^{-1}-1}{C''}\right) &\text{ if } \gep\le (1+C'')^{-1}
 \end{array}\right.
 \end{equation}

\subsection{Upper bound on $\Tm$}

To give an upper bound on the mixing time, we bound the following quantity
\begin{equation}
\bar d (t)= \max_{x,y\in[0,k]^2}\|\bar\mu^{x}_t-\bar \mu^{y}_t\|\ge d(t).
\end{equation}
We define a coupling of two replicas of the  Markov chain $W$ starting from different states as follows: let $(W^{(1)}_t,W_t^{(2)})$ be the Markov chain on $\{1,\dots,k\}^2$
given by the following transition
\begin{equation}\begin{split}
\bP((i,j),(i\pm 1,j))&:=\bar P(i,i\pm 1), \text{ if } i\ne j, \\
\bP((i,j),(i,j\pm 1))&:=\bar P(i,j\pm 1), \text{ if } i\ne j, \\
\bP((i,i),(i\pm 1,i\pm 1))&:= \bar P((i,i\pm 1), \\
\bP((i,j),(i,j))&:= \bar P(i,i)+\bar P(j,j)-1,\\
\bP((i,i),(i,i))&:= \bar P(i,i),\\
\end{split}\end{equation}
where all the other transitions have zero probability. One can check that the coefficients are positive and that this
 indeed defines a stochastic matrix.
This coupling has the property that once $W^{(1)}$  and $W^{(2)}$ merge, they stay together. Moreover, before the merging, at most one of the 
two coordinates changes at each time step, and therefore the sign of $W^{(1)}-W^{(2)}$ 
is constant in time (\textit{i.e.} $W^{(1)}$  and $W^{(2)}$ cannot cross without merging).
With our choice for initial condition, in the sequel it is always non-negative.

If one denotes by $\bP^{x,y}$ the law of $(W^{(1)}_t,W_t^{(2)})$ with initial condition $(x,y)$, $x\ge y$ and defines
\begin{equation}
 \tau:=\inf\{t>0\ | \ W^{(1)}_t=W_t^{(2)}\},
\end{equation}
then
\begin{equation}
 \|\bar\mu^{x}_t-\bar \mu^{y}_t\|\le \bP^{x,y} \left[\tau>t\right].
\end{equation}
Set $D_t:=W^{(1)}_t-W^{(2)}_t\ge 0$. According to Lemma \ref{expec}
\begin{equation}
 \bE^{x,y}[D_t]=(x-y)\left(1-\frac{2}{n}\right)^{t}.
\end{equation}
A first moment analysis is enough to treat the case $\lim_{n\to\infty} k(n)/\sqrt{n}=l$.
Indeed for any $x$ and $y$
\begin{equation}
\bP^{x,y} \left[\tau>t\right]= \bP^{x,y} \left[D_t \ge 1 \right]\le  \bE^{x,y}[D_t]\le l \sqrt{n}(1+o(1))\left(1-\frac{2}{n}\right)^{t}.
\end{equation}
Setting $t_{\gb}=\frac{1}{4}n\log n+ \gb n$ one gets 
\begin{equation}
 \bP^{x,y} \left[\tau>t_{\gb}\right]= l e^{-2\gb}(1+o(1)).
\end{equation}
and therefore
\begin{equation}
 \Tm(\gep)\le \frac{1}{4}n\log n-\frac{n}{2}\log (\gep/l).
\end{equation}

The rest of the section is therefore devoted to the case $\lim_{n\to \infty} k(n)/\sqrt{n}= \infty$.

Given that $W^{(1)}_t>W^{(2)}_t$, $D_t$ has the following transition probabilities
\begin{equation}\begin{split}
 \bP(D_{t+1}=D_{t}-1\ | \ W_t^{(1)}, W_t^{(2)} )&=\bP(W^{(1)}_t,W^{(1)}_t-1)+\bP(W^{(2)}_t,W^{(2)}_t+1)\\ &\quad =2\frac{W_t^{(1)}(n-2k+W^{(1)}_t) +(k-W^{(2)}_t)^2}{n^2},\\
 \bP(D_{t+1}=D_{t}+1\ | \  W_t^{(1)}, W_t^{(2)})&=\bP(W^{(1)}_t,W^{(1)}_t+1)+\bP(W^{(2)}_t,W^{(2)}_t-1)\\ &\quad=2\frac{(k-W^{(1)}_t)^2+ W_t^{(2)}(n-2k+W^{(2)}_t)}{n^2}. 
\end{split}\end{equation}
One can check that the evolution of $D$ is not Markovian (it depends on the values of $W_t^{(1)}$ and $W_t^{(2)}$ and not only on $D_t$). This makes the analysis of $\tau$ difficult.
\medskip
 We now sketch the method we use to tackle this problem:
\begin{itemize} 
 \item First, we use a first moment method to show that after a time $t_0:=\frac{n}{4}\log n$, with probability close to one
 $W_{t_0}^{(1)}-W_{t_0}^{(2)}$ is of order $k/\sqrt{n}$.
 \item Then, we do a sequence of stochastic comparisons to state that starting from $W_{t_0}^{(1)}-W_{t_0}^{(2)}\le k/\sqrt{n}$, $\tau$ is
stochastically dominated by ($t_0$ plus) the hitting time of zero for a simple symmetric random walk on $\Z$ with jump rate $k^2/n^2$ .
 \item Finally, we use a reflection argument to show that the typical time for hitting zero when starting from $k/\sqrt{n}$ for such
a walk is of order $n$. Altogether this gives that $\tau$ is smaller than $t_0+Kn$ with probability close to one if $K$ is sufficiently large.
\end{itemize}
The idea of combining a first moment method and diffusion in order to evaluate mixing times has already been used to compute the mixing time of the mean field Ising model in 
\cite{LLP}, the method was then refined in \cite{DLP2}. The interest of the method here lies in the particular manner the coupling has to be constructed.

Set $t_{\alpha}= \frac{n}{4}\log n+\alpha n$. One uses Markov property to get the following bound on $\tau$.

\begin{multline}
 \bP^{x,y} \left(\tau>t_{\alpha+1}\right)\le \bE^{x,y}\left[ \bP^{W^{(1)}_{t_\ga},W^{(2)}_{t_\ga}}\left(\tau> n\right)\right]\\
\le \bP^{x,y}[W^{(1)}_{t_\ga}-W^{(2)}_{t_\ga}\le M]+ \max_{x'\ge y', x'-y'\le M} \bP^{x',y'}[\tau>n]\\
\le \frac{\bE^{x,y}[W^{(1)}_{t_\ga}-W^{(2)}_{t_\ga}]}{M}+ \max_{x'\ge y', x'-y'\le M} \bP^{x',y'}[\tau>n].
\end{multline}
We apply it for $M=\frac{k e^{-\alpha}}{\sqrt{n}}$. As we have
\begin{equation}
\bE^{x,y}[W^{(1)}_{t_\ga}-W^{(2)}_{t_\ga}]=(x-y)\left(1-\frac{2}{n}\right)^{t_{\alpha}}\le \frac{k}{\sqrt{n}}e^{-2\alpha},
\end{equation}
this gives
\begin{equation}\label{orib}
 \bar d (t_{\alpha+1})\le e^{-\alpha}(1+o(1))+ \max_{x'\ge y', x'-y'\le \frac{k e^{-\alpha}}{\sqrt{n}}} \bP^{x',y'}[\tau>n],
\end{equation}
and all that remains to do is estimating the second term.
The next step is to show that 
\begin{equation}\label{comparf}
 \bP^{x',y'}[\tau>n]\le Q^{x'-y'}[\tau'>n]
\end{equation}
where $\tau'$ is the first hitting time of zero for a nearest neighbor symmetric random walk on $\bbZ$ with ``jump rate'' 
$n/k^2$, starting from $x'-y'$ (law $Q^{x'-y'}$).
The result is rather intuitive, as $D_t=W^{(1)}_t-W^{(2)}_t$ has a drift towards zero and the probability of jumping is bounded from below
by $k^2/n^2$. However, stochastic comparisons have to be made with some care in order to prove the result rigorously. We construct a coupling explicitly.

We define $J_i$ the sequence of moving time for $(W^{(1)}, W^{(2)})$ \textit{i.e.} $J_0=0$ and for $i\ge 0$
\begin{equation}
 J_{i+1}:=\inf\{t\ge J_{i}\ | \ (W^{(1)}_t, W^{(2)}_t)\ne (W^{(1)}_{J_i}, W^{(2)}_{J_i}) \}.
\end{equation}
We remark that $(W^{(1)}_{J_i}, W^{(2)}_{J_i})_{i\ge 0}=(\bar W^{(1)}_i, \bar W^{(2)}_i)_{i\ge 0}$ 
is itself a Markov chain (something similar to the skeleton of a continuous time/discrete space Markov chain), 
and that conditionally to $(W^{(1)}_{J_i}, W^{(2)}_{J_i})_{i\ge 0}$, $(J_{i+1}-J_{i})_{i\ge 0}$, it is 
a sequence of geometric variables of mean $(2-\bar P(W^{(1)}_{J_i},W^{(1)}_{J_i})-\bar P(W^{(2)}_{J_i},W^{(2)}_{J_i}))^{-1}$ 
if $W^{(1)}_{J_i}\ne W^{(2)}_{J_i}$, and of mean $(1-\bar P(W^{(1)}_{J_i},W^{(1)}_{J_i}))^{-1}$ if the processes have merged.

Let $(U_i)_{i\ge 0}$ and $(U'_i)_{i\ge 0}$ be two independent sequences of i.i.d.\ random variables. One constructs the process 
$(W^{(1)}_t, W^{(2)}_t)_{t\ge 0}$ starting from $(x,y)$ deterministically from the sequences $(U_i)_{i\ge 0}$ and 
$(U'_i)_{i\ge 0}$ as follows (for the sake of simplicity, we do not give details of how the construction is done after the merging as 
we do not use it):

\begin{itemize}
 \item First, one constructs recursively $(\bar W^{(1)}_i, \bar W^{(2)}_i)_{i\ge 0}$,   $(\bar W^{(1)}_0, \bar W^{(2)}_0)= (x,y)$ and
\begin{equation}
(\bar W^{(1)}_{i+1}, \bar W^{(2)}_{i+1}):=\left\{
    \begin{array}{ll}
        (\bar W^{(1)}_{i}+1, \bar W^{(2)}_{i})& \text{ if } U'_{i+1} \in [0,a(\bar W^{(1)}_{i},  W^{(2)}_{i})]  \\
        (\bar W^{(1)}_{i}, \bar W^{(2)}_{i}-1)& \text{ if } U'_{i+1} \in (a(\bar W^{(1)}_{i},  W^{(2)}_{i}),b(\bar W^{(1)}_{i},  W^{(2)}_{i}))] \\
	(\bar W^{(1)}_{i}-1, \bar W^{(2)}_{i})& \text{ if } U'_{i+1} \in (b(\bar W^{(1)}_{i},  W^{(2)}_{i}),c(\bar W^{(1)}_{i},  W^{(2)}_{i}))] \\
        (\bar W^{(1)}_{i}, \bar W^{(2)}_{i}+1)& \text{ if } U'_{i+1} \in (c(\bar W^{(1)}_{i},  W^{(2)}_{i}),1]
    \end{array}\right.
\end{equation}
where $a(i,j)$, $b(i,j)$, $c(i,j)$ are chosen such that the chain has the right transition probabilities. It is important to notice 
that $b(i,j)\le 1/2$ for all $i,j$. 
\item Then, given $(\bar W^{(1)}_i, \bar W^{(2)}_i)_{i\ge 0}$ we construct the sequence of moving times: $J_0=0$ and 
\begin{equation}
  J_{i+1}-J_{i}=m \   \text{ if } \  U_{i+1} \in \left[1-(1-q( \bar W^{(1)}_{i}, 
\bar W^{(2)}_{i}) )^{m-1},1-(1-q(\bar W^{(1)}_{i-1}, \bar W^{(2)}_{i-1})^{m}\right),
\end{equation}
where $q(i,j):= (2-\bar P(i,i)-\bar P(j,j))\ge k^2/n^2$ is the inverse of the mean jumping time from $(i,j)$.  
\end{itemize}
Using the same variables $(U_i)_{i\ge 0}$ and $(U'_i)_{i\ge 0}$, one constructs a simple random walk on $\Z$.
\begin{itemize}
 \item First, one defines $\bar X_i$ as
\begin{equation}
 \bar X_i=(x-y)+\sum_{j=1}^i \ind_{\{U'_i\le 1/2\}}-\ind_{\{U'_i>1/2\}}.
\end{equation}
\item Then, one defines $H_i$ by $H_0=0$ and
\begin{equation}
 H_{i+1}-H_{i}=m \text{ if }  U_{i+1} \in \left[1-(1-k^2/n^2, 
 )^{m-1},1-(1-k^2/n^2)^{m}\right).
\end{equation}
\item Finally, one sets
\begin{equation}
 X_t=\bar X_i  \text{ if }  t \in [H_i,H_{i+1}). 
\end{equation}
\end{itemize}
From this construction one has that $H_i\ge J_i$ and $\bar X_i \ge  \bar W^{(1)}_{i}- \bar W^{(2)}_{i}$ for all $i$.
Therefore
\begin{equation}
\tau':=\inf\{t\ |\ X_t=0\}=H_{\inf\{i\ | \ \bar X_i=0\}}\ge J_{\inf\{i\ |   \bar W^{(1)}_{i}= \bar W^{(2)}_{i}\}}=\tau.
\end{equation}
By construction $(X_t)_{t\ge 0}$ is a random walk with transition probability $p(x, x\pm 1)=k^2/(2n^2)$ and $p(x,x)=1-k^2/n^2$, and therefore 
we proved \eqref{comparf}. 
We now finish the proof of the main theorem.
From \eqref{comparf} and \eqref{orib}
\begin{equation}
 \bar d (t_{\alpha+1})\le e^{-\alpha}(1+o(1))+ Q^{\lceil\frac{ k e^{-\alpha}}{\sqrt{n}}\rceil}[\tau'>n],
\end{equation}
where  $Q^{m}$ is the law of $(X_t)_{t\ge 0}$ starting from $m$. 
Then from Proposition \ref{ipsos}
\begin{equation}
Q^{\lceil\frac{ k e^{-\alpha}}{\sqrt{n}}\rceil}[\tau'>n]\le e^{-\alpha}(1+o(1))
\end{equation}
and therefore
\begin{equation}
 \Tm(\gep)\le \frac{n}{4}\log n + n\log(\gep/2)+o(n). 
\end{equation}
Using the same technique one can get that there exists a function $c(\gep)$ that goes to infinity when $\gep$ goes to $0$ such that
\begin{equation}
 \Tm(1-\gep)\le \frac{n}{4}\log n - c(\gep) n+o(n).
\end{equation}
(this, together with the lower-bound part shows that the window of size $O(n)$ is an optimal result \eqref{optimo}).

\subsection{Diffusion bounds}

One is left with proving the approximation we used for the law of $\tau'$.
Let $n$ be a fixed integer.
Let $(X_t)_{t\ge 0}$ be a nearest-neighbor random walk on $\Z$ with transitions $p(x,x\pm1)=q(n)/2$, $p(x,x)=1-q(n)$.
We start our random walk from $\alpha s(n)$ where $s(n)$ and $q(n)$ satisfy:
\begin{equation}
s(n)^2=q(n)n.
\end{equation}
We denote $Q^x$ the probability associated to this random walk starting from $x$ and $Q=Q^0$.
We want to estimate $\tau'$, the hitting time of zero for this random walk.

\begin{proposition}\label{ipsos}
 One has for any given positive $\alpha$ and $\gb$
\begin{equation}
 \lim_{n\to\infty} Q^{\alpha s(n)}\left[\tau' > \gb n \right]=\frac{1}{\sqrt{2\pi}}\int_{ \left[-\frac{\alpha}{\sqrt{\gb}}\ ,\
\frac{\alpha}{\sqrt{\gb}}\right]}e^{-\frac{ s^2}{2}}\dd s\le \frac{\alpha}{\sqrt{\gb}}.
\end{equation}
\end{proposition}

In fact the proof can almost be reduced to proving the following lemma.

\begin{lemma}
For any positive integer $m$ and $n$,
\begin{equation}
 Q^m\left[\tau' > n \right]= Q\left[X_{n}\in [-m+1,m] \right]
\end{equation}
\end{lemma}
\begin{proof}
 We have
\begin{equation}
Q^m \left[\tau > n \right]=\sum_{j>0} Q^m\left(\forall i\in [1,n-1]\ X_i>0, X_n=j\right).
\end{equation}
Moreover,
\begin{multline}
 Q^m\left(\forall i\in [1,n-1]\ X_i>0, X_n=j\right)\\
=Q^m\left(X_n=j\right)-Q^m\left(\exists i\in [1,n-1]\, X_i=0,\, X_n=j\right)\\
=Q^m\left(X_n=j\right)-Q^m\left[X_n=-j\right]=Q\left[X_n=j-m\right]-Q\left[X_n=j+m\right].
\end{multline}
The second inequality just comes from the application of the reflection principle.
Therefore
\begin{equation}
\sum_{j>0} Q^m\left(\forall i\in [1,n-1] X_i>0, X_n=j\right)=Q\left( X_n \in [-m+1,m]\right).
\end{equation}
\end{proof}

With the previous lemma, all one has to do is prove the convergence of $X_{\gb n}/\sqrt{\gb n q(n)}$ to a Gaussian variable.
We do it by computing the Fourrier transform.
For any fixed $K$ we have that uniformly for all $|t|\le K$

\begin{multline}
Q\left[e^{it\frac{X_{\gb n}}{\sqrt{\gb n q(n)}}}\right]
=\left[1-q(n)\left(1-\cos\frac{t}{\sqrt{\gb n q(n)}}\right)\right]^{\gb n}\\
=\left[1-q(n)\frac{t^2}{2 \gb n q(n)}(1+o(1))\right]^{\gb n}=e^{-\frac{t^2}{2}}(1+o(1)).
\end{multline}
Therefore
\begin{multline}
 \lim_{n\to\infty} Q^{\alpha s(n)}\left[\tau' > \gb n \right]= \lim_{n\to\infty} Q\left[\frac{X_{\gb n}}{\sqrt{n q{n}}}\in \frac{[-\alpha s(n)+1,\alpha s(n)]}{\sqrt{\gb n q(n)}}  \right]\\
=P \left[\mathcal N \in \left[-\frac{\alpha}{\sqrt{\gb}}\ ;\ \frac{\alpha}{\sqrt{\gb}}\right]\right].
\end{multline}
where $\mathcal N$ (with law denoted by $P$) is a standard normal variable.

\bigskip

{\bf Acknowledgment:} The authors are very grateful to Pietro Caputo for his constant scientific and technical support, and to
 François Simenhaus for several enlightening discussions. This work has been carried out in the Department of Mathematics of the University of 
Roma Tre during H.L's postdoc and R.L.'s research internship. They gratefully acknowledge the kind hospitality of the department, the support of the Applied Mathematics Department of the Ecole Polytechnique (for R.L.) and the support of ERC Advanced Research Grant ``PTRELSS'' (for H.L.).

%
%
%

\end{document}